\documentclass[11pt,a4paper,twoside]{article}
\usepackage{amssymb,amsthm}
\usepackage[namelimits,sumlimits,fleqn]{amsmath}
\usepackage{setspace}
\usepackage[a4paper,top=33mm, bottom=31mm, left=21mm, right=21mm]{geometry}
\usepackage[small, margin=20pt]{caption}
\usepackage{float}\restylefloat{figure}
\numberwithin{equation}{section}
\usepackage{url}
\usepackage{hyperref}

\usepackage{fancyhdr}
\allowdisplaybreaks[4]

\title{Upper Bounds on the Cardinality of Higher Sumsets}
\author{Giorgis Petridis}
\date{}

\theoremstyle{plain}
\newtheorem{theorem}{Theorem}[section]
\newtheorem{lemma}[theorem]{Lemma}
\newtheorem{proposition}[theorem]{Proposition}
\newtheorem{corollary}[theorem]{Corollary}

\theoremstyle{definition}

\newtheorem{indexedexample}[theorem]{Example}

\theoremstyle{definition}
\newtheorem*{remark}{Remark}

\newcommand{\rst}[1]{\ensuremath{{\mathbin\upharpoonright}%
\raise-.5ex\hbox{$#1$}}} 

\DeclareMathOperator{\im}{Im} 

\begin{document}

\pagenumbering{arabic}

\setcounter{section}{0}


\bibliographystyle{acm}

\maketitle


\begin{abstract}
Let $A$ and $B$ be finite sets in a commutative group. We bound $|A+hB|$ in terms of $|A|$, $|A+B|$ and $h$. We provide a submultiplicative upper bound that improves on the existing bound of Imre Ruzsa by inserting a factor that decreases with $h$.
 \end{abstract}

\section[Introduction]{Introduction}
\label{Introduction}

One of the core problems in additive number theory is to obtain estimates on the cardinality of sumsets. Given sets $A$ and $B$ in a commutative group the sumset of $A$ and $B$ is defined by $$A+B=\{a+b:a\in A, b\in B\}.$$ 

In this paper we are concerned with obtaining upper bounds on the cardinality of sumsets of the from $A+hB$ recursively defined by $A+hB=(A+(h-1)B)+B$. It is easy to check that under no further restriction the extremal examples are when $A$ and $B$ are disjoint sets consisting of generators of a free commutative group. 

Usually in additive number theory the sets $A$ and $B$ are not generic. Very often a bound on $|A+B|$ is known: $|A+B|\leq \alpha |A|$ for some $\alpha\in \mathbb{R}^+$ that could depend on $A, B$. Given the trivial lower bound $|A+B|\geq |A|$ this extra condition measures how much adding $B$ to $A$ changes the cardinality. The question we address is how much adding $B$ repeatedly to $A$ changes the cardinality: we suppose that $A$ and $B$ are finite sets in a commutative group and that both $|A|$ and $|A+B|$ are given and ask for an upper bound on $|A+hB|$ in terms of $|A|$ and $|A+B|$.  

The special case when $A=B$ has attracted most attention in the literature and the answer to our question is well understood. Helmut Pl\"unnecke established in \cite{Plunnecke1970} that $|A+A|\leq \alpha |A|$ implies 
\begin{eqnarray}\label{hA}
|hA|\leq \alpha^h |A|.
\end{eqnarray}
The upper bound is sharp when $A$ is a group and $\alpha=1$. More importantly it has the correct dependence on $\alpha$ and $|A|$: for infinitely many $\alpha\in \mathbb{Q}^+$ there are examples (natural generalisations of Theorem 9.5, Chapter 1 of \cite{Ruzsa2009}) where $|hA|=c(h) \alpha^h |A|$. In these example $c(h)$ is of the order $h^{-h}$ and so there is reason to believe that the dependence on $h$ in Pl\"unnecke's upper bound can be improved when $\alpha$ is large. 

A particular feature of \eqref{hA} is the multiplicativity of the upper bound. By this we mean that replacing $A$ by its $r$-fold tensor product gives the same inequality. This is because $|A|$ is replaced by $|A|^r$, $\alpha$ by $\alpha^r$ and $|hA|$ by $|hA|^r$. 

On the other hand we can get the correct dependence on $h$, and in particular a submultiplicative upper bound, by not insisting on having the best possible power dependence on $\alpha$. Imre Ruzsa has shown in \cite{Ruzsa1999} that $|A+A|\leq \alpha |A|$ implies $$|hA|\leq \alpha^2 \binom{\alpha^4+h-2}{h-1} |A| .$$

The outlook changes when a different set $B$ is added repeatedly to $A$. Ruzsa has studied the problem of bounding $|A+2B|$ in terms of $|A|$ and $|A+B|$ thoroughly. He has shown (Section 6 in \cite{Ruzsa2006}  and Theorem 9.1 of Chapter 1 in \cite{Ruzsa2009}) that $$|A+2B|\leq \alpha^2 |A|^{3/2}.$$  The most significant difference with the $A=B$ case is that the exponent of $|A|$ is no longer one. One may initially suspect that the upper bound must therefore not be sharp, but Ruzsa has shown otherwise. In \cite{Ruzsa1996} he gave examples (for every positive rational $\alpha$ and infinitely many $|A|$) where $$|A+2B|\geq \left(\frac{\alpha-1}{4}\right)^2 |A|. $$ 

Ruzsa's method works equally well for $h\geq 2$ and yields the multiplicative upper bound $$|A+hB|\leq \alpha^h |A|^{2-1/h}.$$ The upper bound can also be derived from a (more general and more recent) result of Balister and Bollob\'as (Theorem 5.1 in \cite{Balister-Bollobas2007}; for a different proof see Corollary 3.7 in \cite{MMT2008}). 

Ruzsa's upper bound is in the correct order of magnitude in $\alpha$ and $|A|$. We demonstrate this by extending an example of his \cite{Ruzsa2006} to larger $h$.
\begin{indexedexample} \label{Universal Lower bound A+hB}
Let $h$ be a positive integer. There exist infinitely many $\alpha\in \mathbb{Q}^+$ with the following property. For each such $\alpha$ there exist infinitely many $m$ such that one can find finite sets $A$ and $B$ in a commutative group with $|A|=m$, $|A+B|\leq \alpha m$ and $$ |A+hB| \geq (1+o(1))\, \frac{\alpha^h}{h(h+1)^h} m^{2-1/h}.$$
\end{indexedexample}
The $o(1)$ term is $o_{m\rightarrow\infty}(1)$.

Ruzsa also noted that the behaviour of $|A+2B|$ (and in fact of $|A+hB|$) changes when $\alpha$ is close to one. He proved (Theorem 10.1 in Chapter 1 of \cite{Ruzsa2009}) $$|A+2B|\leq \alpha m + \frac{3}{2} \alpha (\alpha-1)|A|^{3/2}$$ for $\alpha\leq 2$. His method works equally well for $h\geq 2$ and gives $$|A+hB|\leq \alpha m + \frac{h+1}{h} \alpha^{h-1}
(\alpha-1)|A|^{2-1/h}.$$ It is not clear whether the stated upper bound has the correct dependence on $\alpha$. Extending an example of Ruzsa \cite{Ruzsa2009} to larger $h$ nonetheless shows that the dependence on $(\alpha-1)$ and $|A|$ is correct. 
\begin{indexedexample} \label{Special Lower bound A+hB}
Let $h$ be a positive integer and $\alpha$ a real in the interval $[1,2]$. For infinitely many $m$ there exist finite sets $A$ and $B$ in a commutative group such that $|A|=m$, $|A+B| \leq (1+o(1))\,\alpha m$ and $$|A+hB| \geq
(1+o(1))\,\left(m+\frac{(\alpha-1)}{h} m^{2-1/h}\right).$$
\end{indexedexample}
The $o(1)$ term is $o_{m\rightarrow\infty}(1)$. The same notation will be used throughout this paper.

The main goal of this paper is to improve Ruzsa's upper bounds by introducing a further term that decreases with $h$. This is a first step towards determining the correct dependence of $|A+hB|$ on $h$. We prove the following result.
\begin{theorem} \label{A+hB}
Let $h$ be a positive integer, $\alpha$ a positive real number and $m$ an arbitrarily large integer. Suppose that $A$, $B$ are finite non-empty sets in a commutative group that satisfy $|A|=m$ and
$|A+B| \leq \alpha m$. Then
\begin{eqnarray*}
|A+hB| \leq (1+o(1))\, \frac{e}{2 h^2} \, \alpha^h \hspace{1 pt}
m^{2-1/h}.
\end{eqnarray*}
We also have
\begin{eqnarray*}
|A+hB| \leq m + (1+o(1))\,\frac{e}{h}(\alpha-1)\alpha^{h-1}
m^{2-1/h},
\end{eqnarray*}
which is stronger for $\alpha\leq 1+1/(2h-1)$.

The $o(1)$ term tends to zero as $m$ gets arbitrarily large and is of the order $O(m^{-1/h})$.
\end{theorem}

The biggest qualitative improvement comes from having a term that decreases with $h$ while keeping the optimal dependence on $\alpha$ and $m$. The bound is furthermore submultiplicative, a sharp contrast to many results in the area when different sets are added to one another \cite{Bukh2008,GMR2008,GMR2010}. It should be noted that while it is easy to deduce a multiplicative upper bound from a supermultipicative upper bound, it is not easy to turn a multiplicative bound to submultiplicative. The former task can be done by applying the tensor product trick, which has been applied by Ruzsa and others in many occasions. We will not discuss it any further. A good summary of how powerful it is can be found in \cite{TaoBlogTensor}.

As we will see roughly speaking one a factor of $h$ is saved by strengthening Pl\"unnecke's graph-theoretic method and another factor of $h$ by replacing it by a more efficient elementary counting argument. 

The distinction between the values of $\alpha$ is essential as in the latter case the difference between $\alpha$ and $\alpha -1$ can be substantial. For example $m$ is the dominant term in the second upper bound for $\alpha\leq 1+h e^{-h} m^{-1+1/h}$. 

Another observation is that setting $B=A$ in Theorem~\ref{A+hB} works better than applying \eqref{hA} when $\alpha\geq m^{1-1/h}$ (for example when $A$ consists of generators of a commutative group).

The paper is organised as follows. In Section~\ref{Plunneckes Inequality} we strengthen Pl\"unnecke's graph theoretic method, a task which has interest in its own right. Sections \ref{Ruzsa Upper Bound} to \ref{Upper Bounds} are devoted to motivating and presenting a proof of Theorem~\ref{A+hB}. In section  \ref{Lower Bounds}  Examples \ref{Universal Lower bound A+hB} and \ref{Special Lower bound A+hB} are constructed. In Section~\ref{Further} we state some graph-theoretic results about that follow by a similar approach, but do not provide proofs. Finally in Section~\ref{SetAddition} we discuss how the material in this paper relates with more recent advances in the subject.

\subsection*{Acknowledgements}
The research leading to the results of the paper was done while the author was at the University of Cambridge. The author would like to thank Imre Ruzsa for detailed comments that improved the quality of the paper. In particular the proofs of Propositions~\ref{Restricted addition graph growth} and \ref{A+hB - all alpha} were simplified considerably. Suggestions of Tim Gowers, Ben Green and Peter Keevash have also helped make the paper better. 

\section[Pl\"unnecke's Inequality]{Pl\"unnecke's Inequality}
\label{Plunneckes Inequality}

We begin by recalling Pl\"unnecke's graph theoretic method and
explaining the refinement necessary to obtain Theorem~\ref{A+hB}.
Much of the material in this section can be found in any of the
standard references \cite{Nathanson1996,Ruzsa2009,Tao-Vu2006}. The
notation used is however slightly different.

$G$ will always be a directed layered graph with edge set $E(G)$
and vertex set $V(G)=V_0\cup\dots\cup V_h$, where the $V_i$ are
the \textit{layers} of the graph. For any $S\subseteq V_i$ we
write $S^c= V_i \setminus S$ for the complement of $S$ in $V_i$
and not in $V(G)$. We furthermore assume that directed edges exist
only between $V_i$ and $V_{i+1}$.

We are interested in a special class of such graphs which satisfy a graph-theoretic version of commutativity, the so-called Pl\"unnecke's conditions.  \textit{Pl\"{u}nnecke's upward condition} states that if $uv$ and $vw_i\in E(G)$ for $1\leq i\leq k$, then there exists a vertex $v_i$ for all $1\leq i \leq k$ such that both $uv_i$ and $v_iw_i\in E(G)$. \textit{Pl\"{u}nnecke's downward condition} states that if $vw$ and $u_iv\in E(G)$ for $1\leq i\leq k$, then there exists a vertex $v_i$ for all $1\leq i \leq k$ such that both $u_iv_i$ and $v_iw\in E(G)$. $G$ is called a \textit{commutative graph} when it satisfies both conditions.

The most typical example is $G_+(A,B)$, the \textit{addition
graph} of two sets $A$ and $B$ in an ambient commutative group.
This is defined as the directed graph whose layers are $V_0=A$
and, for all $i>1$, $V_i$ is $A+iB$. A directed edge exists
between $x\in V_{i-1}$ and $y\in V_i$ if and only if $y-x\in B$.

A \textit{path} of length $l$ in $G$ is a sequence of vertices
$v_1,\ldots, v_l$ so that $v_{i}v_{i+1}\in E(G)$ for all $1\leq i
\leq l-1$. For any subgraph $H$ of $G$ we define $
\im^{(i)}_{H}(Z)$ to be the collection of vertices that can be
reached from $Z$ via paths of length $i$ in $H$. When the subscript is omitted
we are taking $H$ to be $G$ and when the superscript is omitted we
are taking the neighbourhood of $Z$ in $H$.

For $i>j$ and $U\subseteq V_i$ $V\subseteq V_j$ the graph consisting of all paths in $G$ starting
at $U$ and ending in $V$ is called a channel. A crucial observation we will use repeatedly is that any channel of a commutative graph is a commutative graph in its own
right. For $ Z\subseteq V_0$ the \textit{channel of} $Z$ is the graph consisting of all paths in $G$ starting
at $Z$ and ending in $V_h$. It should be noted that in this case $\im_H(v) = \im_G(v)$ holds for all $v\in V(H)$.

Ruzsa introduced \textit{restricted addition graphs}, which are
addition graphs with a component removed. Given any three sets
$A$, $B$ and $C$ we take $G_{\!R}(A,B,C)$ to be the graph with
layers $V_0=A$ and $V_i = (A+iB) \setminus (C+(i-1)B)$ for all
$i>0$. The edges between layers are determined similarly to
addition graphs: $xy\in E(V_i,V_{i+1})$ if and only if $y-x\in
B$. $G_{\!R}(A,B,C)$ therefore consists of all the paths in
$G_+(A,B)$ that end in $(A+hB)\setminus (C+(h-1)B)$ and is therefore (a channel and in particular) a commutative graph.

For $i=1,\dots, h$ the $i$th magnification ratio of $G$ is defined as
$$ D_i(G) = \min_{\emptyset\neq Z\subseteq V_0} \frac{|\im^{(i)}(Z)|}{|Z|} .$$
Pl\"unnecke established in \cite{Plunnecke1970} the following.
\begin{theorem}[Pl\"{u}nnecke]\label{Plunnecke}
Let $G$ be a commutative graph. Then the sequence $D_i^{1/i}(G)$
is decreasing.
\end{theorem}
Other proofs can be found in \cite{Ruzsa1989,GPPlIn}. The standard
application of the inequality highlights how powerful it is.
\begin{corollary}\label{hB} 
Let $A$ and $B$ be finite sets in a commutative group and $h$ a positive integer. Suppose that  that $|A|=m$ and $|A+B|\leq\alpha m$. Then $$|hB|\leq D_1(G_+(A,B))^h m \leq \alpha^h\,m.$$
\end{corollary}
\begin{proof}
We work in the addition graph $G=G_+(A,B)$. We know from Theorem
\ref{Plunnecke} that
$$D_h(G)\leq D_1(G)^h\leq \alpha^h$$ and so there is a non-empty $X\subseteq V_0=A$ such
that $$|X+hB| = |\im^{(h)}(X)| \leq D_1(G)^h |X|\leq
\alpha^h\,m.$$ The claim follows as $|hB|\leq |X+hB|$.
\end{proof}
It should be noted that no information is given on the subset of
$V_0$ which gives rise to $D_i(G)$. The first step towards the proof of Theorem
\ref{A+hB} is to strengthen the inequality and prove that any
$Z\subseteq V_0$ which satisfies the property $|\im^{(j)}(Z)|=D_j(G) |Z|$ exhibits
restricted growth.
\begin{theorem}\label{Stronger Plunnecke}
Let $G$ be a commutative graph with vertex set $V_0\cup\dots\cup V_h$. Suppose that $D_j(G)=|V_j| / |V_0|$. Then $$ |V_j|^{h} \geq  |V_0|^{h-j} |V_h|^j.$$ In particular $D_1(G)=|V_1|/|V_0|$ implies
\begin{eqnarray*}
|V_h|\leq \left \lfloor \frac{|V_1|^h}{|V_0|^{h-1}}\right \rfloor = \left \lfloor D_1(G)^h |V_0| \right \rfloor.
\end{eqnarray*}
\end{theorem}
\begin{proof}
Suppose not. Let $G$ be a counterexample where $|V_0|$ is minimum.
Pl\"{u}nnecke's inequality implies that the collection $$\{Z
\subseteq V_0 : |\im^{(h)}(Z)| \leq D_j^{h/j}(G) \, |Z| \}$$ is
nonempty. 

Let $S\subsetneq V_0$ be a set of maximal cardinality in the
collection and $H$ be the channel consisting of paths that start in $S^c$ and end in
$\im^{(h)}(S)^c$. Suppose that $U_0\cup U_1\cup\dots\cup U_h$ are the layers of $H$. 

$U_j$ does not intersect $\im^{(j)}(S)$ as there would
then exist a path in $H$ leading to $\im^{(h)}(S)$. We therefore
have $|U_0| = |V_0| - |S|$ and $|U_j|\leq|V_j|-|\im^{(j)}(S)| \leq
|V_j| - D_j(G) |S| = |V_j| (1-|S|/|V_0|) = |V_j|
(|V_0|-|S|)/|V_0|$. Consequently
\begin{eqnarray}\label{D_j of S complement in Stronger Plunnecke}
 D_j(H)\leq |U_j|/|U_0|\leq |V_j|/|V_0|=D_j(G).
\end{eqnarray}
Let $T\subseteq U_0$ be minimal subject to $|\im_H^{(j)}(T)| =
D_j(H)\, |T|$. Let us get a lower bound on $|\im^{(h)}_H(T)|$. We
know from the maximality of $|S|$ that
\begin{eqnarray*}
D_j^{h/j}(G) \,|S\cup T| & < & |\im^{(h)}(S\cup T)| \\
                   & =  & |\im^{(h)}(S)| + |\im^{(h)}(T) \setminus \im^{(h)}(S)|  \\
                   & =  &  |\im^{(h)}(S)|+|\im^{(h)}_H(T)|  \\
                   & \leq  & D_j^{h/j}(G)\, |S| + |\im^{(h)}_H(T)|.
\end{eqnarray*}
This implies
\begin{eqnarray}\label{D_h of T in Stronger Plunnecke}
 |\im^{(h)}_H(T)|>D_j^{h/j}(G)\, |T|.
\end{eqnarray}
Finally we consider $H^\prime$, the channel of $T$ in $H.$ This is a commutative graph with layers $T_0 \cup\dots\cup
T_h$ and by the defining properties of $T$  and \eqref{D_j of S
complement in Stronger Plunnecke}
\begin{eqnarray}\label{D_j of T in Stronger Plunnecke}
|T_j|/|T_0| = D_j(H) \leq D_j(G).
\end{eqnarray}
By combining \eqref{D_h of T in Stronger Plunnecke} and \eqref{D_j
of T in Stronger Plunnecke} we get:
\begin{eqnarray*}
|T_0|^{h-j}|T_h|^j> (D_j(G)\, |T_0|)^h \geq  |T_j|^h.
\end{eqnarray*}
Thus $H^\prime$ is another counterexample. However, $|T_0|=|T|
\leq |S^c| < |V_0|$, which contradicts the minimality of $|V_0|$.
\end{proof}

\begin{remark}
It is shown in \cite{GPPlIn} that the upper bound is best possible.
\end{remark}
A disadvantage of the traditional form of Pl\"unnecke's inequality
is that it doesn't specify the subset of $V_0$ that exhibits
restricted growth at level $i$. In addition it leaves the
possibility open that different subsets need to be considered for
different $i$. One can get round both difficulties by selecting
any $Z\subseteq V_0$ that satisfies $|\im(Z)|=D_1(G) |Z|$ and
applying Theorem~\ref{Stronger Plunnecke} to the channel of $Z$. It
follows that $|\im^{(i)}(Z)|\leq D_1^i(G) |Z|$ for all $i=1,\dots, h$.

This is in fact the way we will apply the theorem: partition the vertices of $G$
in commutative subgraphs where the condition of the theorem is
satisfied. 
\begin{lemma}\label{Graph partition}
Let $G$ be a commutative graph with vertex set $V_0\cup\dots\cup
V_h $. $V_0$ can be partitioned into $Z_1,\dots,Z_k$ and the vertices of $G$
into vertex disjoint commutative subgraphs $G_1,\dots,G_k$ such
that
\begin{enumerate}
\item $Z_i$ is the bottom layer of $G_i$,
\item $\alpha_i:= D_1(G_i)$ is a strictly increasing sequence,
\item $|\im_{G_i}(Z_i)| = D_1(G_i)\,|Z_i|$.
\end{enumerate}
\end{lemma}
\begin{proof}
We select the commutative subgraphs $G_i$ as follows. We let
$G_1^\star=G$ and $Z_1 \subseteq V_0$ of maximal cardinality subject to $|\im_{G_1^\star}(Z_1)| = D_1(G_1^\star)\, |Z_1|$. We then
define $G_1$ to be the channel of $Z_1$ in $G_1^\star$ and
$\alpha_1=D_1(G_1)=D_1(G_1^\star)$.

We repeat this process in $G_2^\star$, the channel consisting of all paths in $G_1^\star$ that start in $Z_1^c$ and
end in $\im_{G_1}^{(h)}(Z_1)^c$. The layers of $G_1$ and
$G_2^\star$ do not intersect. We select $Z_2 \subseteq Z_1^c$ of
maximal cardinality subject to $|\im_{G_2^\star}(Z_2)| =
D_1(G_2^\star)\,|Z_2|$. We then take $G_2$ to be the channel of $Z_2$
in $G_2^\star$ (and not in $G$) and $\alpha_2 = D_1(G_2)
=D_1(G_2^\star)$. We carry on until $V_0$ is partitioned into
$Z_1\cup \dots \cup Z_k$. Consequently we get a partition of the vertices of $G$
into vertex disjoint commutative subgraphs $G_1,\dots,G_k$.

The sequence $\{\alpha_i\}$ is strictly increasing as the maximality of the
$Z_i$ implies
\begin{eqnarray*} \alpha_i (|Z_i|+|Z_{i+1}|) & < & |\im_{G_i^\star}(Z_i\cup Z_{i+1})| \\
                                             & = & |\im_{G_i}(Z_i)|+|\im_{G_{i+1}}(Z_{i+1})| \\
                                             & = & \alpha_i
                                             |Z_i|+\alpha_{i+1}|Z_{i+1}|. \qedhere
\end{eqnarray*}
\end{proof}
Ruzsa combined Pl\"unnecke's inequality with some other elementary estimates in a clever way to bound $|A+hB|$. The next section is devoted to explaining Ruzsa's method and motivating the proof of Theorem~\ref{A+hB}.  The proof itself, found in sections \ref{Restricted Addition Graphs} and \ref{Upper Bounds}, is entirely self contained.

\section[Ruzsa's Upper Bound]{Ruzsa's Upper Bound}
\label{Ruzsa Upper Bound}

Let us begin by stating again the results one gets by Ruzsa's method.
\begin{theorem}[Ruzsa] \label{Ruzsa bounds}
Let $A$ and $B$ be finite sets in a commutative group and $h$ a positive integer. Suppose such that
$|A|=m$ and $|A+B| \leq \alpha m$. Then
\begin{eqnarray}\label{Ruzsa's universal upper bound}
|A+hB| \leq \alpha^h m^{2-1/h}.
\end{eqnarray}
For $\alpha\leq 2$
\begin{eqnarray*}\label{Ruzsa's upper bound for alpha close to one}
|A+hB| &\leq& \alpha m + (\alpha-1) m^2 \sum_{j=2}^h
(1+\tfrac{1}{j}) \alpha^{j-1} m^{-1/j} \nonumber\\ &\leq& m +
(1+o(1))\,(1+\tfrac{1}{h}) \alpha^{h-1} (\alpha-1) m^{2-1/h}.
\end{eqnarray*}
\end{theorem}
What follows is a heuristic presentation of Ruzsa's argument and the means by which we improve it. Our aim is to help the reader keep the bigger picture in mind in the coming sections and not to provide a detailed presentation.

The best introduction may be to reflect on the limitations of
Pl\"unnecke's inequality. They appear clearly in the proof of
Corollary \ref{hB}. In general there is no reason to assume that
the magnification ratio is $\alpha$ or that $|hB|$ is comparable
to $|X+hB|$ or that $|X|$ is comparable to $|A|$. There are
cases when all assertions hold, for example when $A$ is a subgroup
and $B$ consists of points in distinct cosets of $A$, but there is
much to be gained by a more careful analysis. With these remarks
in mind let us turn to Ruzsa's argument. 

We work in $G:=G_+(A,B),$ be the addition graph of $A$ and $B$. The first step is to partition $A$ into $A_1 \cup
A_2$, which can be thought of as the slow and fast expanding parts
of $A$ under addition with $B$. To bound $|A_2+hB|$ we start with
the trivial estimate $|A_2+hB| \leq |A_2|\,|hB|$ and use Corollary
\ref{hB} to bound $|hB|$. The only known fact about $D_1(G)$ is that it is at most $\alpha$, so it tempting to replace $D_1(G)$ with $\alpha$. Note however that when $D_1(G)=\alpha$ Theorem~\ref{Stronger Plunnecke} can be applied and so $|A+hB|\leq \alpha^h m$, which is
small. We can therefore assume that $D_1(G)=\alpha_1<\alpha$. This is the first novel point of our approach.

The second has to do with bounding $|A_1+hB|$. The standard way to do this is to first apply Pl\"unnecke's inequality to $G$ and get $Z_1\subseteq A$ such that $|Z_1+hB|\leq \alpha^h |Z_1|$. Next apply Pl\"unnecke's inequality to the channel of $A\setminus Z_1$ and get $Z_2\subseteq A\setminus Z_1$ such that $|Z_2+hB| \leq (\tfrac{\alpha m}{m-|Z_1|})^h |Z_2|$. Another application of Pl\"unnecke's inequality to the channel of $A\setminus (Z_1\cup Z_2)$ gives $Z_3\subseteq A\setminus (Z_1\cup Z_2)$ such that $|Z_3+hB|\leq (\tfrac{\alpha m}{m-|Z_1|-|Z_2|})^h |Z_3|$. Iterating gives a subset $A_1 = Z_1\cup\dots\cup Z_k \subset A$ which can be made arbitrarily large (subject to being contained in $A$ of course). The cardinality $|A_1+hB|$ can be bounded by $|Z_1+hB|+\dots+|Z_k+hB|$. Ruzsa calls the resulting statement \emph{Pl\"unnecke's inequality for a large subset}.

The method is imbalanced. While $V_0$ is partitioned, the same is not done for $V_1$. This is largely due to the nature
of Pl\"unnecke's inequality, which gives no lower bounds on the
image in $V_1$ of the set that exhibits restricted growth at level $h$. Using Theorem
\ref{Stronger Plunnecke} instead works better. It introduces the
magnification ratio $\alpha_1$ into the calculations, which is
welcomed as $|hB|$ is bounded in terms of $\alpha_1$, and also
helps us partition $V_1$. As a consequence the numerator of the
fractions found near the end of the preceding paragraph gradually reduces. In fact we will show that a factor of
$(\alpha-\alpha_1)$ appears in the main term. As a consequence
$|A_1+hB|$ becomes rather small when $\alpha_1$ is very close to
$\alpha$. On the other hand the contribution coming from
$|A_2+hB|$ becomes larger as $\alpha_1$ increases. The balancing that takes places
is responsible for reducing Ruzsa's bound by a factor of $h^{-1}$.

To save the additional factor of $h^{-1}$ we have to find a more efficient way to study the growth of $A_2$ than Pl\"unnecke's inequality. To motivate it we examine an example that is typical of sets $A$ and $B$ where $A+hB$ grows fast. Suppose that $A$ is a group and $B$ a collection of points in different cosets of $A$. Then $|A+B|= |A|\,|B|$ and so $\alpha=|B|$. Pl\"unnecke's inequality gives $|A+hB|\leq |A|\,|B|^h$, but an elementary counting argument shows that in fact $|A+hB|\leq |A|\,\tbinom{|B|+h-1}{h}$. With a little care one can extend the
counting argument to a method of bounding $V_h$ that works better for the fast growing part of addition graphs than Pl\"unnecke's inequality. 

Before presenting the details of our approach we note that Ruzsa's trick of bounding $|A_2+hB|$ by $|A_2|\,|hB|$ for the ``fast growing'' $A_2$ will be vital as will be the restricted addition graphs he introduced.

\section[Restricted Addition Graphs]{Restricted Addition Graphs}
\label{Restricted Addition Graphs}

As we saw in Section~\ref{Ruzsa Upper Bound} Pl\"unnecke's inequality appears to not always be optimal to study the growth of addition graphs. As noted a much more elementary counting argument sometimes works better. To make the most of this simple observation one needs to at the very least achieve a similar improvement not only for addition graphs, but for the commutative graphs that result once a component has been removed. These are the restricted addition graphs we defined in Section~\ref{Plunneckes Inequality}. Our first step is to prove that the refinement we are suggesting is not hopeless.
\begin{lemma}\label{Plunnecke for restricted addition graphs}
Let $A$, $B$ and $C$ be finite non-empty sets in a commutative group; $G$ be the restricted addition graph $G_{\!R}(A,B,C)$; and $h$ a positive integer. For all $a\in V_0$
$$|\im^{(h)}(a)|\leq \binom{|\im(a)|+h-1}{h}.$$
\end{lemma}
\begin{proof}
The left hand side is the cardinality of the set $(a+hB)\setminus
(C+(h-1)B)$. Suppose that $a+b_1+...+b_h$ is an element of this
set. If $a +b_i$ belonged to $C$ for any $i$, then this element
would also belong to $C+(h-1)B$. This does not happen and
therefore
\begin{eqnarray*}
(a+hB)\setminus (C+(h-1)B) & \subseteq &  \{a+b_1+\dots+b_h:b_i\in B,\;a+b_i\notin C\}\\
                                                                    & =  &  \{a+b_1+\dots+b_h:b_i\in B\setminus(C-a)\}.
\end{eqnarray*}
The left hand side is therefore at most
$$\binom{|B\setminus(C-a)|+h-1}{h} = \binom{|(a+B)\setminus
C|+h-1}{h}, $$ which is the right hand side.
\end{proof}
We use the lemma to partition the vertices of a restricted addition graph much
like we did with Lemma~\ref{Graph partition} and get an estimate
on the cardinality of its layers.

\begin{proposition}\label{Restricted addition graph growth}
Let $A$, $B$ and $C$ be finite non-empty sets in a commutative group,
$G=G_{\!R}(A,\hspace{-1.6pt}B,\hspace{-1.6pt}C)$ and $h$ a positive integer. Define $\beta,$ the \emph{pseudo-cardinality} of $B,$ to be the positive real number that satisfies
\begin{eqnarray*}
\binom{\beta+h-1}{h}= |hB| .
\end{eqnarray*}
Suppose that the layers of $G$ are $V_0\cup\dots\cup V_h$. Then
\begin{eqnarray*}
|V_h| \leq \frac{|V_1| \, |hB|}{\beta} \leq  \left( 1 + \frac{h}{\beta}\right) \frac{e\,|V_1| \, |hB|^{1-1/h}}{h}. 
\end{eqnarray*}

\end{proposition}
\begin{remark}
The sets $A$ and $C,$ which have seemingly disappeared from the conclusion, appear implicitly in the quantity $|V_1|= |(A+B)\setminus (B+C)|.$ 
\end{remark}

\begin{proof}[Proof of Proposition~\ref{Restricted addition graph growth}]
Let $x=|V_0|$ and put an arbitrary order to the elements of $A$ so that $A=\{a_1,\dots, a_x\}$. 

Define a sequence of graphs by $G_1= G_{\!R}(a_1,B, C)$ and, for $i>1,$  $G_i= G_{\!R}(a_i,B, C\cup (\{a_1,\dots,a_{i-1}\} +B )).$ So that, say, for $i>0,j>1$ the $j$th layer of $G_i$ is $(a_i + jB) \setminus (\,(\{a_1,\dots a_{i-1}\} + j B) \cup (C+ (j-1) B)\,).$ The vertex sets of the $G_i$ therefore partition the vertex set of $G$  and so 
\begin{eqnarray*}
|V_j| = \sum_{i=1}^x |\im_{G_i}^{(j)}(a_i)| \mbox{ for $j=0,\dots, h$} .
\end{eqnarray*}

To keep the notation simple we define the quantities $r_i = |\im_{G_i}(a_i)|$ for all $i=1,\dots, x.$ In particular we have that $|V_1| = \sum_{i=1}^x r_i .$ 

Next we observe that 
\begin{eqnarray}\label{min}
|\im_{G_i}^{(h)}(a_i)| \leq \min\left\{ \binom{r_i+h-1}{h} , |hB|  \right\}.
\end{eqnarray}
The inequality following from Lemma~\ref{Plunnecke for restricted addition graphs} and the bound 
\begin{eqnarray*}
|\im_{G_i}^{(h)}(a_i)| \leq |a_i +hB| = |hB| .
\end{eqnarray*}

To bound the minimum in \eqref{min} we observe that
\begin{eqnarray*}
\frac{\binom{r+h-1}{h}}{r} = \sum_{i=0}^{h-1} c_i r^i~\mbox{for positive constants $c_i$ that depend on $h$}.
\end{eqnarray*}
So the function $r \mapsto \tbinom{r+h-1}{h} / r$ is increasing. In particular 
\begin{eqnarray*}
\frac{1}{r} \min\left\{ \binom{r+h-1}{h} , |hB| \right\}  \leq \frac{|hB|}{\beta}.
\end{eqnarray*}
Consequently for all $1\leq i \leq x$ we have:
\begin{eqnarray*}
|\im_{G_i}^{(h)}(a_i)| \leq \min\left\{ \binom{r_i+h-1}{h} , |hB|  \right\} \leq \frac{|hB|}{\beta} \,r_i.
\end{eqnarray*}

Summing over $i=1,\dots,x$ gives 
\begin{eqnarray*}
|V_h|  \leq \sum_{i=1}^x \frac{|hB|}{\beta} r_i = \frac{|V_1| \, |hB|}{\beta}.
\end{eqnarray*}

For the second inequality we observe 
\begin{eqnarray*}
|hB| = \binom{\beta+h-1}{h} \leq \left(\frac{e (\beta+h)}{h}\right)^h 
\end{eqnarray*}  
It follows that  $(\beta+h)^{-1} \leq e h^{-1} |hB|^{-1/h}$ and so 
 \begin{eqnarray*} 
|V_h| \leq \left( 1 + \frac{h}{\beta}\right) \frac{|V_1| \, |hB|}{\beta+h} \leq  \left( 1 + \frac{h}{\beta}\right) \frac{e\,|V_1| \, |hB|^{1-1/h}}{h}. \qedhere
\end{eqnarray*} 
\end{proof}

Balister and Bollob\'as obtained a similar upper bound on $|A+hB|$. It follows from Theorem 5.1 in \cite{Balister-Bollobas2007} that $|V_h| \leq |V_1| |hB|^{1-1/h}.$  The upper bound in Proposition~\ref{Restricted addition graph growth} is better by about a factor of $1/h$ when $h = O(\beta)$.

The upper bound is furthermore sharp. Take $A$ and $B$ to be disjoint sets that consist solely of generators of a free commutative group and $C$ to be the empty set. Then $|V_1|=|A+B| = |A|\,|B|, |V_h| = |A+hB| = |A|\,|hB|$ and $\beta=|B|.$ In other words in the proposition we are essentially establishing that $|V_h|$ is maximum when $B$ consists of points that are independent with respect to addition with $A$. 

It is also worth noting that the upper bound is sharp up to a constant even if $\beta$ is much smaller than $|B|.$ For all $\alpha \in \mathbb{Q}^+$ Ruzsa has constructed examples (Theorem 5.5 in \cite{Ruzsa2006}) of integer sets $A$ that satisfy $|2A| = \alpha |A|$ and $|3A| \geq c |2A|^3=c \alpha^3 |A|^{3/2}$, for some absolute constant $c>0.$ In this case $\beta$ is up to a constant $\sqrt{|2A|} = \sqrt{\alpha |A|}$ and so the upper bound is up to a constant attained.

Setting $C=\emptyset$ and applying Corollary \ref{hB} gives 
\begin{eqnarray*}
|A+hB|\leq \left(1+\frac{h}{\beta}\right) \, \frac{e}{h} \,\alpha^h m^{2-1/h}. 
\end{eqnarray*} 
For the purpose of Theorem~\ref{A+hB} we can assume that $\beta$ tends to infinity with $|A|$. This is because  $\tbinom{\beta+h-1}{h} = |hB|$ can be taken to be at least $ |A|^{1/3}$ (otherwise $|A+hB|\leq |A| |hB| \leq |A|^{2-2/3}$) and $h$ is assumed to be a constant.

So Lemma~\ref{Plunnecke for restricted addition graphs} can be used to improve Theorem~\ref{Ruzsa bounds}. Lemma~\ref{Graph partition} also leads to a similar upper bound on $|A+hB|$. We will not show how this is done, but only present a sketch for the benefit of the reader familiar with Ruzsa's paper. In Section~\ref{Further} it discussed how Lemma~\ref{Graph partition} leads to a stronger form of Pl\"unnecke's inequality for a large subset (the term is defined in Section~\ref{Ruzsa Upper Bound}). Using the resulting Theorem~\ref{Growth large subset} in Ruzsa's proof allows one to treat the magnification ratio $\alpha_1$ of $G_+(A,B)$ as a variable that is not automatically assumed to equal $\alpha$. This subtle change results in the additional factor of $1/h$. The best bound however comes by combining the two lemmata.

\section[Upper Bounds]{Upper Bounds}
\label{Upper Bounds}

To prove Theorem~\ref{A+hB} we will apply Theorem~\ref{Stronger Plunnecke} to the slow growing part of the graph (where the magnification ratio plays a role and thus enters the calculations) and Proposition~\ref{Restricted addition graph growth} to the fast growing part.

\begin{proposition}\label{A+hB - all alpha}
Let $h$ be a positive integer, $\alpha$ a positive real number and $m$ an arbitrarily large integer. Suppose that $A$, $B$ are finite non-empty sets in a commutative group that satisfy $|A|=m$, $|A+B| \leq \alpha m$ and $D_1(G_+(A,B))=\alpha_1$. Then
\begin{eqnarray*}
|A+hB| \leq \frac{e}{h}\alpha_1^{h-1}(\alpha - \alpha_1) m^{2-1/h}+ O\left(\alpha^h m^{2-2/h}\right).
\end{eqnarray*}
In particular
\begin{eqnarray*}
|A+hB| \leq \frac{e}{h^2} \left(1-\tfrac{1}{h}\right)^{h-1} \alpha^h \hspace{1pt} m^{2-1/h} + O\left(\alpha^h m^{2-2/(h+1)}\right).
\end{eqnarray*}
\end{proposition}

\begin{proof}
We begin with some preliminary considerations. We set 
\begin{eqnarray*}
s = \frac{|hB|}{\beta}
\end{eqnarray*} 
where the pseudocardinality $\beta$ of $B$ defined in the statement of Proposition~\ref{Restricted addition graph growth}. If $s \leq \alpha^{h-1}$, then we are done as by Proposition~\ref{Restricted addition graph
growth} $|A+hB| \leq s |A+B| \leq \alpha^h m$. So from now on we assume that $\alpha\leq s^{\tfrac{1}{h-1}}$.

Next we apply Lemma~\ref{Graph partition} and get a partition of $A$ into $Z_1\cup\dots\cup Z_k$ and a resulting partition of the vertices of $G$ into the vertices of a sequence of graphs $G_1,\dots , G_k$.  It follows that
\begin{eqnarray}\label{size of A+hB}
|A+hB| =  \sum_{i=1}^k |\im^{(h)}_{G_i}(Z_i)|. 
\end{eqnarray}

To estimate this sum we chose an index $j\in\{1,\dots, k\}$. The value of $j$ will be determined later. Applying Theorem~\ref{Stronger Plunnecke} for $1\leq i \leq j$ gives
\begin{eqnarray*}
|(Z_1\cup\dots\cup Z_j) + hB| = \sum_{i=1}^j |\im^{(h)}_{G_i}(Z_i)| \leq \sum_{i=1}^j \alpha_i^h |Z_i|.
\end{eqnarray*}
To bound the size of $(A+hB) \setminus ((Z_1\cup\dots\cup Z_j) + hB)$ we apply Proposition~\ref{Restricted addition graph growth} with $C = (Z_1\cup\dots\cup Z_j) + B$:
\begin{eqnarray*}
|(A+hB) \setminus ((Z_1\cup\dots\cup Z_j) + hB)| & \leq &  s |(A+B) \setminus ( (Z_1\cup\dots\cup Z_j) + B)| \\
									    &    =   & s \sum_{i=j+1}^k \alpha_i |Z_i| , 
\end{eqnarray*}
It is therefore clear that the optimal cutting point $j$ is the largest index for which $\alpha_j \leq s^{\tfrac{1}{h-1}}$. Note also that from the opening remarks we can assume that $\alpha_1 \leq \alpha \leq  s^{\tfrac{1}{h-1}}$.

Equation \eqref{size of A+hB} now becomes: 
\begin{eqnarray}\label{size of A+hB with min}
|A+hB| \leq \sum_{i=1}^k \min\{\alpha_i^h , s \alpha_i\} |Z_i|.
\end{eqnarray}

The minimum can be estimated by a linear function as follows.
\begin{lemma}\label{linear}
Let $1\leq i \leq k$. In the notation established above 
\begin{eqnarray*}
\min\{\alpha_i^h , s \alpha_i\} \leq \alpha_i^h + t (\alpha_i -\alpha_1),
\end{eqnarray*}
where 
\begin{eqnarray*}
t = \frac{s^{\tfrac{h}{h-1}} - \alpha_1^h}{s^{\tfrac{1}{h-1}} -\alpha_1}\,.
\end{eqnarray*}
\end{lemma}
\begin{proof}[Proof of Lemma~\ref{linear}]
For $1\leq i\leq j$ the minimum is $\alpha_i^h$. The inequality holds because the function $\alpha\mapsto \alpha^h$ is convex and the quantity $t$ has be chosen so that the linear function equals $\alpha^h$ when $\alpha = \alpha_1$ or $s^{\tfrac{1}{h-1}}$. For $j<i \leq k$ the minimum is $s \alpha_i$ and so we are comparing linear functions that meet at $\alpha = s^{\tfrac{1}{h-1}}$. It is therefore enough to observe that $s\leq t$. This concludes the proof of the lemma.
\end{proof}
Substituting the estimate we get from the above lemma in \eqref{size of A+hB with min}  yields:
\begin{eqnarray}\label{size of A+hB with s}
|A+hB| & \leq &  \sum_{i=1}^k (\alpha_1^h + t (\alpha_i - \alpha_1)) |Z_i| \nonumber \\
	    &   =   &  \alpha_1^h m + t (\alpha-\alpha_1) m \nonumber \\
	    & \leq & \alpha_1^h m +  (\alpha-\alpha_1) m (s+ h \, s^{\tfrac{h-2}{h-1}} \alpha_1).
\end{eqnarray}
In the last inequality we used the assumption that $\alpha_1 \leq s^{\tfrac{1}{h-1}}$. 

Our next task is to bound $s$. The second inequality in Proposition~\ref{Restricted addition graph growth} states
\begin{eqnarray*}
s \leq \left(1+ \frac{h}{\beta}\right) \frac{e \, |hB|^{1-1/h}}{h}.
\end{eqnarray*}
Very much like in the penultimate paragraph of Section~\ref{Restricted Addition Graphs} we can assume that $h$ is fixed and $\beta$ tends to infinity with $m$. It follows that $h/\beta = O(|hB|^{-1/h})$ and consequently that
\begin{eqnarray*}
s \leq \frac{e \, |hB|^{1-1/h}}{h} + O(h^{-1} |hB|^{1-2/h}).
\end{eqnarray*}
Corollary~\ref{hB} gives $|hB| \leq \alpha_1^h m $ and so
\begin{eqnarray*}
s \leq \frac{e \, \alpha_1^{h-1} m^{1-1/h}}{h} + O(\alpha_1^{h-2} h^{-1} m^{1-2/h}).
\end{eqnarray*}

Straightforward calculations give the first inequality:
\begin{eqnarray*}
|A+hB|  \leq \frac{e  (\alpha-\alpha_1) \alpha_1^{h-1}}{h}  m^{2-1/h}  + O(\alpha^h m^{1-2/h}) .
\end{eqnarray*}
The expression is maximised when $\alpha-\alpha_1 = \alpha/h$ and thus 
\begin{eqnarray*}
|A+hB| \leq \frac{e}{h^2} \left(1-\tfrac{1}{h}\right)^{h-1} \alpha^h m^{2-1/h} + O\left(\alpha^h m^{2-2/(h+1)}\right). \qedhere
\end{eqnarray*}
\end{proof}

We can now deduce Theorem~\ref{A+hB}.
\begin{proof}[Proof of Theorem~\ref{A+hB}]

For the first part we observe that the function $h\mapsto (1-1/h)^{h-1}$ is decreasing. For large $h$ the upper bound gets arbitrarily close to 
\begin{eqnarray*}
h^{-2} \alpha^h m^{2-1/h}.
\end{eqnarray*}

For the second part of the theorem, when $\alpha$ is close to one, we prove by induction that
\begin{eqnarray}\label{induction}
|A+hB| \leq \alpha m + (\alpha-1) m \sum_{i=2}^h s_i,
\end{eqnarray}
where $s_i = |iB|/\beta_i$ and $\beta_i$ is defined by $\tbinom{\beta_i+i-1}{i} = |iB|$.

The $h=1$ case is clear. For $h>1$ we consider a different restricted addition graph that was studied by Ruzsa in \cite{Ruzsa2006}. 

We take any $b\in B$ and observe that
\begin{eqnarray*}
|A+hB| = |b+A+(h-1)B| + |(A+hB) \setminus(b+A+(h-1)B)| .
\end{eqnarray*}

To bound the first term observe that $|b+A+(h-1)B|  =   |A+(h-1)B| $. By the induction hypothesis
\begin{eqnarray*}
|A+(h-1)B| \leq  \alpha m + (\alpha-1) m \sum_{i=2}^{h-1} s_i.
\end{eqnarray*}

To bound the second term we apply Proposition~\ref{Restricted addition graph growth} to $G_{\!R}(A,B,b+A)$. The cardinality of $V_1$, the second layer of this restricted addition graph, is $(\alpha -1)m$ and so
\begin{eqnarray*} 
 |(A+hB) \setminus(b+A+(h-1)B)| \leq (\alpha-1) m \, s_h.
\end{eqnarray*}
This completes the proof of \eqref{induction}. To finish the proof of Theorem~\ref{A+hB} we note that
\begin{eqnarray*}
s_i \leq (1+o(1)) \,\frac{e}{h} \alpha^{h-1} m^{1-1/h}.
\end{eqnarray*}
and that replacing the first summand in \eqref{induction} by $m$ makes no difference to the asymptotic value.
\end{proof}

\section[Examples]{Examples}
\label{Lower Bounds}

We now present Examples \ref{Universal Lower bound A+hB} and
\ref{Special Lower bound A+hB}. As noted above they are extensions
of those given by Ruzsa in \cite{Ruzsa2006,Ruzsa2009}. To keep the
notation as simple as possible we will assume that all values are integers as the construction works for sufficiently composite values of the parameters
which make the rational values integer if necessary.

We begin with Example \ref{Universal Lower bound A+hB}. Let $a$ and $l$ be integers, which we consider as variables with $a$ assumed to be arbitrarily large. We let $b=l a$ and fix $h$. We will work in $\mathbb{Z}_b^k$, where $k= h + a^{h-1}/h$. We write $x_i$ for the $i$th coordinate of the vector $x$.

We consider $A = A_1 \cup A_2$ where $$A_1= \{x :
x_i\in\{0,l,2l,\dots,(a-1)l\}~\mbox{for}~1\leq i\leq
h~\mbox{and}~x_i=0~\mbox{otherwise}\}$$ and $A_2$ is a collection
of $a^{h-1}/h$ independent points $$ A_2= \bigcup_{j=h+1}^k\{x : x_i =
\delta_{ij}~\mbox{for all $i$}\}.$$ $B$ is taken to be a collection
of $h$ copies of $\mathbb{Z}_b$ $$B = \bigcup_{j=1}^h \{x : 1\leq
x_i \leq b\, \delta_{ij}~\mbox{for all $i$}\}.$$ We estimate the cardinality of the
sets that interest us.
$$|A|= a^h + a^{h-1}/h =(1+o(1)) a^h.$$ As $h$ is fixed different
values of $a$ result to different values of $m$. To get an upper
bound on $|A+B|$ we note that
$$\left|A_1+B\right|\leq\sum_{j=1}^h|A_1+\{x : 1\leq
x_i \leq b\, \delta_{ij}\}|\leq h b a^{h-1}$$ and that $$
|A_2+B|\leq |B|\,|A_2|\leq h b a^{h-1}/h = b a^{h-1}.$$ Thus
\begin{eqnarray*}
|A+B|  &\leq& |A_1+B|+|A_2+B|\\
       &\leq& h b a^{h-1}+ b a^{h-1}\\
       & =  & (h+1) l a^{h}\\
       & =  &  (1+o(1))\,(h+1) l m.
\end{eqnarray*}
$\alpha$ is therefore about $(h+1)l$. $h$ is fixed and so
different values of $l$ result in different $\alpha$.

To bound $|A+hB|$ from below observe that $|hB|=b^h$ and that for
$a,a^\prime\in A_2$ the intersection $(a+hB)\cap(a^\prime+hB)$ is
trivial. Thus
\begin{eqnarray*}
|A+hB| &\geq& |A_2+hB|\\
       & =  & 1+(b^h-1)a^{h-1}/h\\
       & =  & (1+o(1))\, b^h a^{h-1}/h\\
       & =  & (1+o(1))\, l^h a^{2h-1}/h\\
       & =  &  (1+o(1))\,\frac{\alpha^h}{h(h+1)^h} m^{2-1/h}.
\end{eqnarray*}
We are done. We have constructed sets $A$ and $B$ with the desired
property. As $a$ and $l$ assume bigger values so do respectively
$m$ and $\alpha$. In other words the bound of Theorems \ref{A+hB}
and \ref{Ruzsa bounds} is of the correct order of magnitude in
$\alpha$ and $m$.

The difference between Example \ref{Universal Lower bound A+hB}
and Theorem~\ref{A+hB} is huge in terms of $h$. To get a feel of
where the two calculations differ we look back at the proof and
the examine the points where it could be generous. The only such
point where the proof and the example agree is $|A_1+hB|=|hB|$. On
the other hand the greatest disparity appears in the growth of
$|A_1+hB|$. By applying Theorem~\ref{Stronger Plunnecke} we assume
the growth is exponential. This means that $|A_1+hB|$ (and
crucially also $|hB|$) should be in the order of $$\left(\frac{|A_1+hB|}{|A_1|}\right)^h |A_1| = (1+o(1))\left(\frac{hb}{ a}\right)^h |A_1|.$$ In the example however $$|hB| = b^h = (1+o(1)) \left(\frac{b}{ a}\right)^h |A_1|.$$

We now turn to Example \ref{Special Lower bound A+hB}. This time
we fix $1<\alpha\leq 2$ and $h$. We let $a$ be an arbitrarily large integer and set $b=(\alpha-1) a^{h-1}/h$.

We work in a commutative group that has subgroups $B_1,\dots,B_h$ of
cardinality $a$ with pairwise trivial intersection. We take
$$A_1=B_1+\dots+B_h$$ of cardinality $a^h$ and $$A_2=\{a_1,\dots,a_b\}$$ to be a collection of
points lying in distinct non-zero cosets of $A_1$. We set
$$A=A_1\cup A_2$$ and $$B=\bigcup_{i=1}^h B_i.$$

We estimate the cardinality of various sets as before.
$$|A|=a^h+b=(1+o(1)) a^h.$$ Thus different values of $a$
result to different values of $m$. As we are free to chose $a$ we
are free to assign infinitely many values to $m$.

For $A+B$ we observe that $A_1+B=A_1$ and that $|A_2+B|\leq
|A_2|\,|B|\leq b h a$. Thus
\begin{eqnarray*}
|A+B|  &\leq& |A_1+B|+|A_2+B|\\
       &\leq& a^{h}+ b h a\\
       & =  & \alpha a^{h}\\
       & =  &  (1+o(1))\alpha m.
\end{eqnarray*}
For $A+hB$ we observe that $hB=A_1$ and so $A_1+hB=A_1$ and
$A_2+hB$ consists of $|A_2|$ translates of $A_1$. $A+hB$ therefore consists of $|A_2|+1$ translates of $A_1$ whose pairwise intersections
are trivial. We are done as
\begin{eqnarray*}
|A+hB| & =  & 1+((b+1)a^h-1)\\
       & =  &  (1+o(1)) (a^h+ ba^h)\\
       & =  &  (1+o(1)) \left(m+ \frac{(\alpha-1)}{h} m^{2-1/h}\right).
\end{eqnarray*}

\section[Results About Commutative Graphs]{Results About Commutative Graphs}
\label{Further}

In this section we present three further results about general commutative graphs. All three are similar to the results we have obtained thus far and can be proved
by a similar method to the proof of Theorem~\ref{A+hB}: a partitioning of the vertices of the graph (similar to that given by Lemma
\ref{Graph partition} or Lemma~\ref{Plunnecke for restricted addition graphs}) followed by an optimisation process similar
to the proof of Proposition~\ref{Restricted addition graph growth}. 

The first result is strengthening what was earlier referred to as Pl\"unnecke's inequality for a large subset. Often in applications one is not
solely interested in a subset of $V_0$ that exhibits restricted
growth, but in a large subset with this property. A repeated
application of Pl\"unnecke's inequality as described in Section
\ref{Ruzsa Upper Bound} takes care of this (c.f. Corollary 7.1 in
\cite{TaoNotes} and Theorem 3.2 in \cite{Ruzsa2006}). Our method
is a little more efficient.

\begin{theorem}\label{Growth large subset}
Let $G$ a commutative graph with vertex set $V_0\cup\dots\cup V_h$. Suppose that $|V_0|=m$ and $|V_1|=n$. For any $m > t\in \mathbb{R}$ there exists non-empty $X\subseteq V_0$ with $|X|>t$ such that $$ |\im^{(h)}(X)| \leq (|X|-t) \left(\frac{n}{m-t}\right)^h.$$
If we furthermore suppose that $D_1(G)=\alpha_1$, then $$ |\im^{(h)}(X)| \leq \alpha_1^h\,t+ (|X|-t) \left(\frac{n-\alpha_1\,t}{m-t}\right)^h.$$
\end{theorem}

The first inequality is a small improvement over the above
mentioned results. As we have seen the biggest potential gain
comes by introducing the magnification ratio of the graph in the
second inequality. It should be noted that the bound cannot be improved by
much even when we consider addition graphs. As mentioned in Section~\ref{Restricted Addition Graphs}, combining Ruzsa's
argument in \cite{Ruzsa2006} with Theorem~\ref{Growth large
subset} leads to $|A+hB|\ll h^{-1} \alpha^h m^{2-1/h}$.

The second result is the generalisation of Proposition~\ref{Restricted addition graph growth} to general commutative graphs. 

\begin{theorem}\label{Growth commutative}
Let $G$ a commutative graph with vertex set $V_0\cup\dots\cup V_h$. Suppose that $M$ is the maximal cardinality of the images in $V_h$ of one-element sets
$$ 
M = \max_{v\in V_0} |\im^{(h)} (v)| 
$$
and the quantity $\beta$ is given by
$$
M = \binom{\beta +h-1}{h} . 
$$
Then
$$ 
|V_h| \leq \frac{M |V_1|}{\beta} .
$$
\end{theorem}

The theorem can be used as an alternative to the trivial estimate $|V_h| \leq M |V_0|.$ The proof is identical to that of Proposition~\ref{Restricted addition graph growth} with only one difference. Lemma~\ref{Plunnecke for restricted addition graphs} no longer applies. The conclusion $|\im^{(h)}(v)| \leq \tbinom{ |\im(v)|+h-1}{h} $ nonetheless holds for all $v\in V_0$ in general commutative graphs. It can be proved by an inductive argument  (e.g. Lemma 4.4 of \cite{GPPlIn}). The rest of the proof is identical to that of Proposition~\ref{Restricted addition graph growth}.  

By combining the two preceding theorems one can bound the cardinality of the layers of commutative graphs in terms of the cardinality of the bottom two layers. This is a generalisation of what we have seen so far as $|A+hB|$ is simply the cardinality of the $h$th layer of $G(A,B)$. 

We can contract $V_0$ to a single vertex and get another commutative graph where the cardinality of the rest of the layers remains unchanged. The generalisation of Lemma~\ref{Plunnecke for restricted addition graphs} to commutative graphs implies that $|V_h|\leq {|V_1|+h-1 \choose h}$ holds for all commutative graphs. This upper bound is in fact best possible under no further assumption on $G$ as all but one elements of $V_0$ may have empty image. To eliminate this sort of examples we assume that $D_h(G)$ is non-zero. Even in this case the bound obtained from the contraction is reasonably accurate. It can nonetheless be improved.

\begin{theorem}\label{Growth general m n h}
Let $G$ a commutative graph with vertex set $V_0\cup\dots\cup V_h$. Suppose that $|V_0|=m , |V_1|=n$ and $D_h(G)>0$. We have $n\geq m^{1-1/h}$ and
$$ |V_h| \leq (1+o(1)) \frac{(n-m^{1-1/h}+3h)^h}{h!}.$$
\end{theorem}
The $o(1)$ term is as usual $o_{m\rightarrow \infty}(1)$. The proof is very similar to the proof of Theorem~\ref{A+hB}. The bound we get is much larger as Ruzsa's trick of bounding the faster growing parts of $V_0$ using Corollary~\ref{hB} no longer applies. It should also be noted that the condition $n\geq m^{1-1/h}$ follows from the assumption on $D_h(G)>0$ and Theorem~\ref{Stronger Plunnecke}. The perhaps mysterious $m^{1-1/h}$ term that appears in the numerator is the minimum value that $\alpha_1 m$ can attain, where as usual $\alpha_1$ is the first magnification ratio of the graph.

The bound is furthermore reasonably sharp. An \textit{independent
addition graph} is $G_+(\{0\},\{\gamma_1,\dots,\gamma_n\})$ where
0 is the identity and $\gamma_1,\dots,\gamma_n$ the generators of
a free commutative group. The \textit{inverse} of a commutative graph
is the commutative graph we obtain by reversing the direction of the paths.
When we consider the union of a suitably chosen independent addition graph and
the inverse of another suitably independent addition graph we see that the
bound in Theorem~\ref{Growth general m n h} cannot be improved
much.

\section[Further Remarks About Sumsets]{Further Remarks About Sumsets}
\label{SetAddition}

We conclude the paper by discussing Theorem~\ref{Stronger Plunnecke} in the context of set addition. Let $A$ and $B$ be finite sets in an abelian group. We wish to apply Theorem~\ref{Stronger Plunnecke} to the addition graph $G_+(A,B)$. Note that in this context $\im^{(i)}(Z)=Z+iB$ for any $Z\subseteq A$. There is no reason why $D_1(G_+(A,B)) = |A+B|/|A|$ and so we pick $\emptyset\neq X \subseteq A$ such that $|X+B|=D_1(G_+(A,B)) |X|$. Applying Theorem~\ref{Stronger Plunnecke} to the addition graph $G_+(X,B)$ (the details can be found below in the proof of Corollary \ref{Restricted Sumset Growth}) gives $$|X+hB| \leq D_1(G_+(A,B))^h |X| = \left(\frac{|X+B|}{|X|}\right)^h |X|.$$ The bound holds for all $h$. The traditional form of Pl\"unnecke's inequality does not guarantee that the same $X$ works for all $h$. The key property of $X$ which allows this is the fact that for all $\emptyset\neq Z \subseteq X$ we have $$\frac{|X+B|}{|X|} \leq \frac{|Z+B|}{|Z|}.$$

This property of the suitably chosen subset $X$ was extended further in \cite{GPNAP}. It was shown there that $X$ has an even stronger property. 
\begin{eqnarray} \label{NAP lemma}
|S+X+B|\leq \frac{|X+B|}{|X|} |S+X| 
\end{eqnarray} 
for any finite set $S$. The inequality can also be extended to not necessarily commutative groups. 

Theorem~\ref{Stronger Plunnecke} has a longer proof than \eqref{NAP lemma} (one has to first establish Pl\"unnecke's inequality), but on the other hand is much more general as it applies to commutative and not just to addition graphs. For example it allows one to work in restricted addition graphs and/or compare $|V_h|/|V_0|$ to $|V_j|/|V_0|$ for any $1\leq j \leq h$. As an illustration we present the following application, which is a variation on Ruzsa's restricted addition graphs.  
\begin{corollary}\label{Restricted Sumset Growth}
Let $h$ a positive integer and $X$, $B$ and $J$ be finite sets in a commutative group with $J\cap X=\emptyset$. Suppose that $$\frac{|(X+jB)\setminus (J+jB)|}{|X|}=\alpha^j $$ for some $1\leq j \leq h$ and that  $$\frac{|(X+jB)\setminus (J+jB)|}{|X|} \leq \frac{|Z+jB)\setminus (J+jB)|}{|Z|} $$ for all $\emptyset\neq Z \subseteq X$. Then $$ |(X+hB)\setminus(J+hB)| \leq \alpha^h |X|.$$  
\end{corollary}
\begin{proof}
Let $H$ be the commutative subgraph of $G_+(X,B)$ that consists of all paths that end in $(X+hB)\setminus(J+hB)$. The layers of $H$ are $V_0=X$ and $V_i=(X+iB)\setminus(J+iB)$ for $i>1$. For $Z\subseteq V_0=A$ we have $\im(Z)=(Z+iB)\setminus(J+iB)$. Thus the condition on $X$ is equivalent to $$D_j(H)=\frac{|V_j|}{|V_0|} =\alpha^j.$$ Applying Theorem~\ref{Stronger Plunnecke} gives the desired bound on $|V_h|=|(X+hB)\setminus (J+hB)|$.  
\end{proof}
Christian Reiher has obtained a generalisation to the corollary in the spirit of \eqref{NAP lemma}. He has shown in \cite{Reiher2011} that under the same assumptions on $X$, $B$, $J$ and $\alpha$ the following inequality holds for all finite sets $S$ $$ |(X+jB+S)\setminus (J+jB+S)| \leq \alpha |(X+(j-1)B+S)\setminus(J+(j-1)B+S)|. $$ The proof is relatively short and purely combinatorial. Corollary \ref{Restricted Sumset Growth} can easily be deduced by induction on $h$ by setting $S=B$. 

Reiher's inequality  could therefore have been used to derive Theorem~\ref{A+hB} instead of the material in Section~\ref{Plunneckes Inequality}. We opted to present the graph theoretic approach as Theorem~\ref{Stronger Plunnecke} may be helpful in other contexts.

\bibliography{all}

$\hspace{12pt}$\textsc{ Department of Pure Mathematics and Mathematical Statistics, Wilberforce Road, Cambridge CB3 0WB, England}

$\hspace{12pt}$ \textit{Email address}: giorgis@cantab.net

\end{document}